\documentclass[12pt,a4paper,leqno]{amsart}

\usepackage[applemac]{inputenc}
\usepackage[T1]{fontenc} 
\usepackage{amsthm}
\usepackage{amsfonts}       
\usepackage{amsmath}
\usepackage{amssymb}
\usepackage{mathrsfs}
\usepackage{xcolor}

\newcommand{\R}{\mathbb{R}}

\newcommand{\Z}{\mathbb{Z}}
\newcommand{\ud}{\mathrm{d}}

\newcommand{\F}{\mathscr{F}}

\newcommand{\D}{\mathscr{D}}

\newcommand{\E}{\mathbb{E}}
\newcommand{\de}{\Delta}

\newcommand{\G}{\mathscr{G}}

\theoremstyle{plain}
\newtheorem{theorem}{Theorem}

\theoremstyle{definition}
\newtheorem{definition}[theorem]{Definition}

\theoremstyle{remark}
\newtheorem{remark}{Remark}

\numberwithin{equation}{section}
\numberwithin{theorem}{section}

\author{Emil Vuorinen}

\title{$L^{p}(\mu) \to L^{q}(\nu)$ characterization for well localized operators}

\address{DEPARTMENT OF MATHEMATICS AND STATISTICS, P.O.B 68 (GUSTAF H\"ALLSTR\"OMIN KATU 2B), FI-00014 UNIVERSITY OF HELSINKI, FINLAND}
\email{emil.vuorinen@helsinki.fi}

\begin{document}

\begin{abstract}
We consider a two weight $L^{p}(\mu) \to L^{q}(\nu)$-inequality for well localized operators as defined and studied by F. Nazarov, S. Treil and A. Volberg \cite{NTV} when $p=q=2$. A counterexample of F. Nazarov shows that the direct analogue of the results in \cite{NTV} fails for $p=q\not=2$. Here a new square function testing condition is introduced and applied to characterize the two weight norm inequality. The use of the square function testing condition is also demonstrated in connection with certain positive dyadic operators.
\end{abstract}



\maketitle

\section{Introduction}

The main purpose of this paper is to consider two weight norm inequalities for ``well localized operators'' (see definition \ref{well loc}). In \cite{NTV} F. Nazarov, S. Treil and A. Volberg proved that Sawyer type testing conditions are necessary and sufficient for a well localized operator $T$ to be bounded from $L^{2}(\mu)$ into $ L^{2}(\nu)$, where $\mu$ and $\nu$ are two arbitrary  Radon measures on $\R^{n}$. This means that to deduce boundedness of the operator $T$ it suffices to test $T$ and its formal adjoint with one indicator of a (dyadic) cube at a time. Here we investigate the same well localized operators but with general exponents  $1<p<\infty$ defining the $L^{p}$-spaces. As an example of the applicability of the two weight theorem for well localized operators it was shown in \cite{NTV} that two weight inequalities for \emph{Haar multipliers} and \emph{Haar shifts} can be seen as two weight inequalities for well localized operators.

There exists a manuscript by F. Nazarov \cite{F} showing that there are situations where the Sawyer type testing conditions do not work in $L^{p}$ when $ p \not=2$. He fixes an exponent $ 1<p<\infty, \ p \not=2$, and provides an example of a certain operator related to Haar multipliers for which the Sawyer type testing conditions with the exponent $p$ do not imply the corresponding quantitative two weight estimate. Also, in this example the Sawyer type testing  would be enough in the case $p=2$. A quantitative consequence of this counterexample related to Haar multipliers is explained in Section \ref{Well localized operators}.

But, if we look at the Sawyer type testing a little differently, we see that there is another way to generalize it to other exponents $1<p< \infty$. Namely, we consider a kind of \emph{square function testing condition}, whose motivation comes from  $\mathcal{R}$-bounded operator families as used for instance in \cite{W}.   An operator family on $L^{2}$-spaces is $\mathcal{R}$-bounded if and only if it is uniformly bounded, but for other exponents $1<p<\infty$ $\mathcal{R}$-boundedness is in general a stronger property.  In the same spirit our square function testing condition is equivalent with the Sawyer type testing in the case $p=2$, but for other exponents $1<p<\infty$ it can be a stronger requirement.

The initial idea was to try if this kind of testing is necessary and sufficient  for a well localized operator $T$ to bounded from $L^{p}(\mu)$ into $L^{p}(\nu)$ for any exponent $1 < p < \infty$, which indeed is the case. But it was observed that another property of this square function testing is that it gives with exactly the same proof also a characterization for $T$ to be bounded from $L^{p}(\mu)$ into $L^{q}(\nu)$ for any exponents $1<p,q<\infty$. 

To see what kind of theorem we are talking about we formulate a simplified qualitative version of the main Theorem \ref{well loc}. For the exact definition of the operator we refer to Section \ref{Well localized operators}.

\begin{theorem}\label{simplified}
Assume we have two exponents $1<p,q<\infty$ and two Radon measures $\mu$ and $\nu$ on $\R^{n}$. Let $T^{\mu}$ be a well localized operator with respect to a dyadic lattice $\D$ in $\R^{n}$, and suppose $T^{\nu}$ is a formal adjoint of $T^{\mu}$. Then the operator $T^{\mu}$ extends to a bounded operator $T^{\mu}:L^{p}(\mu) \to L^{q}(\nu)$ if and only if there exist two non-negative constants $\mathcal{T}$ and $\mathcal{T}^{*}$, so that for every finite subcollection $\D_{0} \subset \D$ and every set of non-negative real numbers $\{a_{Q}\}_{Q \in \D_{0}}$
the inequalities 

\begin{equation}\label{global}
\Big\| \Big( \sum_{Q \in \D_{0}} ( T^{\mu}a_{Q}1_{Q})^{2} \Big)^{\frac{1}{2}} \Big \|_{L^{q}(\nu)} 
\leq
\mathcal{T} \Big\| \Big( \sum_{Q \in \D_{0}} (a_{Q}1_{Q})^{2} \Big)^{\frac{1}{2}} \Big\|_{L^{p}(\mu)}
\end{equation}
and 
\begin{equation}\label{global dual}
\Big\| \Big( \sum_{Q \in \D_{0}} ( T^{\nu}a_{Q}1_{Q})^{2} \Big)^{\frac{1}{2}} \Big\|_{L^{p'}(\mu)} 
\leq
\mathcal{T}^{*} \Big\| \Big( \sum_{Q \in \D_{0}} (a_{Q}1_{Q})^{2} \Big)^{\frac{1}{2}} \Big\|_{L^{q'}(\nu)}
\end{equation}
hold.
 
\end{theorem}
We will also demonstrate the use of our testing condition with positive dyadic operators, and we will get again an $L^{p}(\mu) \to L^{q}(\mu)$ characterization for any exponents $1<p,q<\infty$. Previously there has been two different characterizations depending on the relative order of the exponents $p$ and $q$, see \cite{LSU-T} (or \cite{H} for a different proof technique) and \cite{T}. Here we get one characterization for all cases. This also provides another example of a situation where the square function testing is sufficient but the Sawyer type testing is not.

Even though we have a different kind of testing condition, the proofs will follow the existing outlines. With the positive dyadic operators we follow the technique in \cite{H}, and our study of the well localized operators is structured as in \cite{NTV}.

\subsection*{Acknowledgements}
I am a member of the Finnish Centre of Excellence in Analysis and Dynamics Research. This work is part of my PhD project under supervision of T. P. Hyt\"onen, and I am very grateful for all the key ideas and discussions related to this problem. Also, I would like to thank the referee for comments that made the proof of Theorem \ref{well loc} shorter and clearer.
 
\section{Set up and preliminaries}\label{Set up and preliminaries}

We begin by recalling a general theorem due to  Marcinkiewicz and Zygmund \cite{MZ} (Theorem \ref{vector extension})  showing that bounded linear operators on $L^{p}$- spaces have extensions to a certain vector valued situation. This theorem will also show that the square function testing condition follows from boundedness of the corresponding operator.

So fix a positive integer $n$ and suppose $\mu$ and $\nu$ are two Radon measures on $\R^{n}$. We consider these fixed for the rest of the paper. We shall give the definitions below with the measure $\mu$, but they are defined similarly with any Radon measure. 

Let $(\varepsilon_{i})_{i=1}^{\infty}$ a sequence of independent random signs on some probablity space $(\Omega, \mathbb{P})$. This means that the sequence is independent and $\mathbb{P}(\varepsilon_{i}=1)=\mathbb{P}(\varepsilon_{i}=-1)=1/2$ for all $i$.  We will use the Kahane-Khinchine inequality \cite{K} saying that for any Banach space $X$ and two exponents $1 \leq p,q< \infty$ there exists a constant $C>0$, depending only on $p$ and $q$, so that for any  $x_{1}, \dots, x_{N} \in X$ 

\begin{equation}\label{KK}
C^{-1}\Big(\E \| \sum_{i=1}^{N} \varepsilon_{i}x_{i} \|_{X}^{q} \Big)^{\frac{1}{q}}  
\leq \Big(\E \| \sum_{i=1}^{N} \varepsilon_{i}x_{i} \|_{X}^{p} \Big)^{\frac{1}{p}} \leq C \Big(\E \| \sum_{i=1}^{N} \varepsilon_{i}x_{i} \|_{X}^{q} \Big)^{\frac{1}{q}},
\end{equation}
where $\E$ refers to the expectation with respect to the random signs. The Kahane-Khinchine inequalities will be used when $X=\R$ or when $X$ is some $L^{p}$-space, and we note here that the constant $C$ in (\ref{KK}) in the case of $L^{p}$-spaces does not depend on the underlying measure. 

Two sided estimates like (\ref{KK}) will be abbreviated as 

\begin{equation*}
\Big(\E \| \sum_{i=1}^{N} \varepsilon_{i}x_{i} \|_{X}^{q}\Big)^{\frac{1}{q}}  
\simeq_{p,q} \Big(\E \| \sum_{i=1}^{N} \varepsilon_{i}x_{i} \|_{X}^{p} \Big)^{\frac{1}{p}},
\end{equation*}
where possible subscripts (in this case $p,q$) refer to the information that the implicit constant $C$ depends on. A similar one sided estimate will be abbreviated as	``$\lesssim$'' or ``$\gtrsim$''.  The implicit constants will never depend on any relevant information in the situation, and no confusion should arise.

For simplicity all our scalar valued functions will be real (or $[- \infty, \infty]$) valued. For any  exponent $1\leq p < \infty$ we denote by $L^{p}(\mu)$ the usual $L^{p}$-space on $\R^{n}$ with respect to the measure $\mu$, and by $L^{p}(\mu,l^{2})$
the space of sequences $(f_{i})_{i=1}^{\infty}$ of $\mu$-measurable real valued functions defined on $\R^{n}$ for which the norm
\begin{equation*}
\| (f_{i})_{i=1}^{\infty} \|_{L^{p}(\mu,l^{2})}:=\Big( \int \Big( \sum_{i=1}^{\infty} | f_{i} |^{2}\Big)^{\frac{p}{2}} \ud \mu \Big)^{\frac{1}{p}}
\end{equation*} 
is finite. 

\begin{theorem}\label{vector extension}
Let $1 \leq p,q<\infty$ be two exponents and assume that $T:L^{p}(\mu) \to L^{q}(\nu)$ is a bounded linear operator. Then the operator
\begin{equation*}
(f_{i})_{i=1}^{\infty} \mapsto  \tilde{T}(f_{i})_{i=1}^{\infty}:= (Tf_{i})_{i=1}^{\infty}
\end{equation*}
is also a bounded linear operator from $L^{p}(\mu,l^{2})$ into $L^{q}(\nu,l^{2})$, with operator norm satisfying 
\begin{equation*}
\| T \|_{L^{p}(\mu) \to L^{q}(\nu)} \simeq_{p,q} \| \tilde{T} \|_{L^{p}(\mu, l^{2}) \to L^{q}(\nu, l^{2})}.
\end{equation*}
\end{theorem}

\begin{proof}
We recall a short proof for the reader's convenience. It suffices to consider an arbitrary sequence $(f_{i})_{i=1}^{\infty}$ of $L^{p}(\mu)$-functions such that $f_{i}\not=0$ only for finitely many indices $i$. Let $(\varepsilon_{i})_{i=1}^{\infty}$ be an independent sequence of random signs.  

Using the Kahane-Khinchine  inequality (four times) and the linearity of $T$ we get
\begin{equation*}
\begin{split}
&\Big( \int \big( \sum_{i=1}^{\infty} |T f_{i} |^{2}\big)^{\frac{q}{2}} \ud \nu \Big)^{\frac{1}{q}} 
= \Big( \int \big( \E| \sum_{i=1}^{\infty} \varepsilon_{i} Tf_{i} |^{2}\big)^{\frac{q}{2}} \ud \nu \Big)^{\frac{1}{q}} \\
& \simeq_{q} \Big( \E \int  | \sum_{i=1}^{\infty} \varepsilon_{i} Tf_{i} |^{q} \ud \nu \Big)^{\frac{1}{q}} 
\simeq_{q} \E \Big(  \int  | T\sum_{i=1}^{\infty} \varepsilon_{i} f_{i} |^{q} \ud \nu \Big)^{\frac{1}{q}} \\
& \leq \| T \|_{L^{p}(\mu) \to L^{q}(\nu)} \E \Big(  \int  | \sum_{i=1}^{\infty} \varepsilon_{i} f_{i} |^{p} \ud \mu \Big)^{\frac{1}{p}}
\simeq_{p}\| T \|_{L^{p}(\mu) \to L^{q}(\nu)} \Big( \int \big( \sum_{i=1}^{\infty} | f_{i} |^{2}\big)^{\frac{p}{2}} \ud \mu \Big)^{\frac{1}{p}},
\end{split}
\end{equation*}
where in the last step we used the Kahane-Khinchine inequality twice.

\end{proof}

Let $\D$ be a dyadic lattice in $\R^{n}$. More specifically, for each $k \in \Z$, let $\D_{k}$ consist of disjoint cubes of the form $x+[0, 2^{-k})^{n}$, $x \in \R^{n}$, that cover $\R^{n}$. It is required that for every $k \in \Z$ and $Q \in \D_{k}$  the cube $Q$ is a union of $2^{n}$ cubes $Q' \in \D_{k+1}$.  Then define $\D:= \cup_{k \in \Z}\D_{k}$.  The side length $2^{-k}$ of a cube  $Q \in \D_{k}$ is  written as $l(Q)$. We will fix one lattice $\D$.

For a cube $Q \in \D_{k}$ define $Q^{(1)}$ to be the unique cube in $\D_{k-1}$ that contains $Q$, and for $2 \leq r \in \Z$ define inductively $Q^{(r)}:= (Q^{(r-1)})^{(1)}$. Also, for any positive integer $r$, let $ch^{(r)}(Q)$ consist of those cubes $Q'$ in $\D$ that satisfy $Q'^{(r)}=Q$, and for $r=1$ write just $ch(Q):=ch^{(1)}(Q)$. We talk about $ch(Q)$ as the children of the cube $Q$.

Let $f$ be a function in $L^{p}(\mu), 1 < p< \infty$. For any cube $Q \in \D$ denote the average of $f$ over $Q$ by
\begin{equation*}
\langle f \rangle^{\mu}_{Q}:= \frac{1}{\mu(Q)} \int_{Q} f \ud \mu, 
\end{equation*}
and define the differences
\begin{equation*}
\de^{\mu}_{Q} f:= \sum_{Q' \in ch(Q)}  \langle f \rangle^{\mu}_{Q'} 1_{Q'}- \langle f \rangle^{\mu}_{Q} 1_{Q}.
\end{equation*}
We shall use the martingale difference decomposition 
\begin{equation*}
f = \sum_{Q \in \D_{k}} \langle f \rangle^{\mu}_{Q} 1_{Q}
+\sum_{\begin{substack}{Q \in \D \\ l(Q) \leq 2^{-k}}\end{substack}} \de^{\mu}_{Q} f,
\end{equation*}
where $k \in \Z$ is any integer. 

For any cube $Q \in \D$ with at least two children that have non-zero $\mu$-measure, let $h^{\mu}_{Q, k}, k \in \{1, \dots, m(Q)\}$, be a collection of Haar functions on $Q$, where $m(Q)+1$ is the number of children of $Q$ that have non-zero measure. The Haar functions are required to form an orthonormal basis for the space
\begin{equation}\label{space}
\{f:Q \to \R: f\text{ is constant on the children of $Q$ and } \int f \ud \mu=0\}
\end{equation}
equipped with the $L^{2}(\mu)$-norm. Below we shall sometimes just write $h^{\mu}_{Q}$ for a generic Haar function related to a cube $Q\in\D$. 

With the Haar functions the differences $\de^{\mu}_{Q}f$ may be written as
\begin{equation*}
\de^{\mu}_{Q}f= \sum_{k=1}^{m(Q)} \langle f, h^{\mu}_{Q,k} \rangle_{\mu} h^{\mu}_{Q,k},
\end{equation*}
where
\begin{equation*}
\langle f,g\rangle_{\mu}:= \int f g \ud \mu
\end{equation*}
for $g \in L^{p'}(\mu)$, and $p'$ is the dual exponent to $p$, i.e., $\frac{1}{p} + \frac{1}{p'} =1$. Indeed, if $Q$ has at most one child with non-zero measure, then $\de^{\mu}_{Q}f=0$. Otherwise, the requirement that Haar functions are constant in the children of $Q$ and have zero average, and the fact that every function in the space (\ref{space}) can be represented with the basis, give
\begin{equation*}
\begin{split}
&\sum_{Q' \in ch(Q)}  \langle f \rangle^{\mu}_{Q'} 1_{Q'}- \langle f \rangle^{\mu}_{Q} 1_{Q}
= \sum_{k=1}^{m(Q)} \Big\langle \sum_{Q' \in ch(Q)} \langle f \rangle^{\mu}_{Q'} 1_{Q'}- \langle f \rangle^{\mu}_{Q} 1_{Q}, h^{\mu}_{Q,k} \Big\rangle_{\mu} h^{\mu}_{Q,k} \\
&=\sum_{k=1}^{m(Q)} \Big\langle \sum_{Q' \in ch(Q)} \langle f \rangle^{\mu}_{Q'} 1_{Q'}, h^{\mu}_{Q,k} \Big\rangle_{\mu} h^{\mu}_{Q,k}
=\sum_{k=1}^{m(Q)} \langle f, h^{\mu}_{Q,k} \rangle_{\mu} h^{\mu}_{Q,k}.
\end{split}
\end{equation*}

The norm of $f$ may be estimated with the martingale difference decomposition as
\begin{equation}\label{mart. norm}
\| f \|_{L^{p}(\mu)}  
\simeq_{p} \Big\| \big( \sum_{Q \in \D_{k}} | \langle f \rangle^{\mu}_{Q} |^{2} 1_{Q}
+\sum_{\begin{substack}{Q \in \D: \\ l(Q) \leq 2^{-k}}\end{substack}} | \de^{\mu}_{Q} f |^{2}\big)^{\frac{1}{2}} \Big\| _{L^{p}(\mu)},
\end{equation}
where again $k \in \Z$ is arbitrary. We emphasize that 
\begin{equation}\label{only for p=2}
\| f \|_{L^{2}(\mu)} =  \Big(\sum_{Q \in \D_{k}} \| \langle f \rangle^{\mu}_{Q} 1_{Q}\|_{L^{2}(\mu)}^{2}
+\sum_{\begin{substack}{Q \in \D \\ l(Q) \leq 2^{-k}}\end{substack}} \| \de^{\mu}_{Q} f \|_{L^{2}(\mu)}^{2}\Big)^{\frac{1}{2}}
\end{equation}
holds only for $p=2$, and in general if one replaces all the numbers $2$ in (\ref{only for p=2}) with an arbitrary $1<p<\infty$, one gets (\ref{only for p=2}) with  ``$\lesssim$'' if $1<p\leq2$ and with ``$\gtrsim$'' if $2 \leq p < \infty$. 

\subsection{Principal cubes and Carleson embedding theorem}\label{Principal and Carleson}

We shall also use the usual principal cubes and a form of the dyadic \emph{Carleson embedding theorem}. To construct the principal cubes suppose $f \in L^{1}_{loc}(\mu)$ and  $\D_{0} \subset \D$ is a subcollection such that each $Q' \in \D_{0}$ is contained in some maximal cube $Q\in \D_{0}$. Maximality of a cube here means that it is not contained in any strictly bigger cube in the same collection.   

Let $\F_{0}$ be the set of maximal cubes in $\D_{0}$. Assume that $\F_{0}, \dots, \F_{k}$ are defined for some non-negative integer $k$. Then, for $Q \in \F_{k}$,  let $ch_{\F}(Q)$ consist of the maximal cubes $Q' \in \D_{0}$ such that $Q' \subset Q$ and
\begin{equation*}
\langle |f| \rangle_{Q'}^{\mu} > 2 \langle |f| \rangle_{Q}^{\mu}.
\end{equation*}
Set $\F_{k+1}:= \cup_{Q \in \F_{k}} ch_{\F}(Q)$ and 
\begin{equation*}
\F:= \bigcup_{k=0}^{\infty}\F_{k}.
\end{equation*}

For any cube $Q \in \D_{0}$ denote by $\pi_{\F}Q$ the smallest cube in $\F$ that contains $Q$, and by $\pi^{1}_{\F}Q$ the smallest cube (if it exists)  in $\F $ that strictly contains it. 

The collection $\F$ is $\frac{1}{2}$-sparse, that is, there exist pairwise disjoint subsets $E(F) \subset F, F \in \F,$ such that $\mu(E(F)) \geq \frac{1}{2}\mu(F)$. Indeed, one can define $E(F):= F \setminus \cup_{F' \in ch_{\F}(F)}F'$, and the construction of $\F$ implies that $\mu(E(F)) \geq \frac{1}{2}\mu(F)$. The property that $\F$ is $\frac{1}{2}$-sparse implies also that $\F$ is $\emph{2-Carleson}$, i.e., for every $F \in \F$ 
\begin{equation*}
\sum_{\begin{substack}{F' \in \F : \\ F' \subset F}\end{substack}} \mu(F') \leq 2 \mu(F).
\end{equation*}

The well known Carleson embedding theorem says that if $\{a_{Q}\}_{Q\in \D}$ is a collection of non-negative real numbers, then the estimate
\begin{equation*}
\sum_{Q \in \D} |\langle f \rangle_{Q}^{\mu} |^{p}a_{Q} \leq C \| f \|_{L^{p}(\mu)}^{p}
\end{equation*}
holds for all $f \in L^{p}(\mu)$, where $C$ is a fixed constant and $p \in (1,\infty)$, if and only if  there exists $C'>0$ so that
\begin{equation*}
\sum_{\begin{substack}{Q' \in \D: \\
Q' \subset Q}\end{substack}} a_{Q'} \leq C' \mu(Q)
\end{equation*}
for all $Q \in \D$.

The version of the theorem we shall use is the following:

\begin{theorem}\label{Carleson}
Suppose $\D_{0} \subset \D$ is a subcollection and $1 < p < \infty$. Then we have the estimate 
\begin{equation}\label{embedding}
\Big\| \sum_{Q \in \D_{0}} \langle | f |\rangle^{\mu}_{Q}1_{Q}\Big\|_{L^{p}(\mu)}
\leq C \|f \|_{L^{p}(\mu)}
\end{equation}
for all $f \in L^{p}(\mu)$, where $C$ is independent of $f$, if and only if there exists $C'>0$ such that for all $Q \in \D_{0}$

\begin{equation}\label{assumption}
\sum_{\begin{substack}{Q' \in \D_{0}: \\ Q' \subset Q}\end{substack}} \mu(Q') \leq C' \mu(Q).
\end{equation}

Moreover, the smallest possible constants $C$ and $C'$ satisfy $C' \leq C^p \lesssim_p C'^{p}$.
\end{theorem}

\begin{proof}
Assume (\ref{embedding}) holds. Then for any $Q \in \D_{0}$ we have
\begin{equation*}
\sum_{\begin{substack}{Q' \in \D_{0}: \\ Q' \subset Q}\end{substack}} \mu(Q')
= \int \sum_{\begin{substack}{Q' \in \D_{0}: \\ Q' \subset Q}\end{substack}} (\langle 1_{Q}\rangle^{\mu}_{Q'}1_{Q'})^{p} \ud \mu
\leq \int (\sum_{\begin{substack}{Q' \in \D_{0}: \\ Q' \subset Q}\end{substack}} \langle 1_{Q}\rangle^{\mu}_{Q'}1_{Q'})^{p} \ud \mu \leq C^p \mu(Q).
\end{equation*}

On the other hand assume that (\ref{assumption}) holds and $f \in L^{p}(\mu)$. If $g \in L^{p'}(\mu)$ is any function, then
\begin{equation*}
\begin{split}
&\int (\sum_{Q \in \D_{0}} \langle | f |\rangle^{\mu}_{Q}1_{Q}) g \ud \mu
= \sum_{Q \in \D_{0}} \langle | f |\rangle^{\mu}_{Q} \langle  g \rangle^{\mu}_{Q} \mu(Q) \\
& \leq \Big( \sum_{Q \in \D_{0}} (\langle | f |\rangle^{\mu}_{Q})^{p}\mu(Q) \Big)^{\frac{1}{p}} 
\Big( \sum_{Q \in \D_{0}} (\langle | g| \rangle^{\mu}_{Q})^{p'}\mu(Q) \Big)^{\frac{1}{p'}}
\lesssim_{p}C'\|f \|_{L^{p}(\mu)} \|g \|_{L^{p'}(\mu)},
\end{split}
\end{equation*}
where the last step follows from the usual formulation of the Carleson embedding theorem.

\end{proof}

\section{Positive dyadic operators}

Before going to work with the well localized operators we introduce and illustrate the square function testing condition with a simpler positive dyadic operator. Fix two exponents $1 < p,q < \infty$. Let $\{\lambda_{Q}\}_{Q \in \D}$ be a set of non-negative real numbers. Define for non-negative Borel measurable functions a mapping   
\begin{equation}\label{def of pos}
f \mapsto T^{\mu}f:=\sum_{Q \in \D} \lambda_{Q} \int_{Q} f \ud \mu1_{Q}.
\end{equation}
We want to investigate when we have an estimate
\begin{equation}\label{estimate pos}
\|T^{\mu}f \|_{L^{q}(\nu)} \leq C \|f\|_{L^{p}(\mu)}
\end{equation}
for all $0 \leq f \in L^{p}(\mu)$, where of course $C$ should not depend on $f$. Similarly define for  $0 \leq f$
\begin{equation*}
f \mapsto T^{\nu}f:=\sum_{Q \in \D} \lambda_{Q} \int_{Q} f \ud \nu1_{Q},
\end{equation*}
and for every cube $Q \in \D$ also the localized versions
\begin{equation*}
T_{Q}^{\mu}f:=\sum_{\begin{substack}{Q' \in \D: \\ Q' \subset Q}\end{substack}} \lambda_{Q'} \int_{Q'} f \ud \mu1_{Q'}
\end{equation*}
and correspondingly $T_{Q}^{\nu}$.

\begin{theorem}\label{Pos oper}
The estimate (\ref{estimate pos}) holds if and only if there exist two constants $0\leq C_{1}, C_{2} < \infty$ so that for every 
finite 2-Carleson family $\D_{0} \subset \D$ and every set $\{a_{Q}\}_{Q \in \D_{0}}$ of positive real numbers the inequalities
\begin{equation}\label{direct pos}
\Big\| \Big( \sum_{Q \in \D_{0}} (a_{Q}T^{\mu}_{Q}1_{Q})^{2}\Big)^{\frac{1}{2}} \Big\|_{L^{q}(\nu)}
\leq C_{1}  \Big\| \Big( \sum_{Q \in \D_{0}} (a_{Q}1_{Q})^{2}\Big)^{\frac{1}{2}} \Big\|_{L^{p}(\mu)}
\end{equation}
and
\begin{equation}\label{dual pos}
\Big\| \Big( \sum_{Q \in \D_{0}} (a_{Q}T^{\nu}_{Q}1_{Q})^{2}\Big)^{\frac{1}{2}} \Big\|_{L^{p'}(\nu)}
\leq C_{2} \Big\| \Big( \sum_{Q \in \D_{0}} (a_{Q}1_{Q})^{2}\Big)^{\frac{1}{2}} \Big\|_{L^{q'}(\mu)}
\end{equation}
hold.

If $\mathcal{T}$ and $\mathcal{T}^{*}$ denote the smallest possible constants $C_{1}$ and $C_{2}$, respectively, then the smallest possible constant $\|T\|$ in (\ref{estimate pos}) satisfies $\|T\| \simeq \mathcal{T} + \mathcal{T}^{*}$.

\end{theorem}

This problem and the related results, as well as the whole ``testing philosophy'', has its roots in the work of E. Sawyer \cite{S1}, \cite{S2} in the 80's. A characterization for the inequality (\ref{estimate pos}) was first given by F. Nazarov, S. Treil and A. Volberg \cite{NTV1} in the case $p=q=2$ using the Bellman function method. The case $1<p\leq q <\infty$ was characterized by M. Lacey, E. Sawyer and I. Uriarte-Tuero in \cite{LSU-T}. Finally, H. Tanaka \cite{T} gave a characterization when the exponents are in the order $1<q<p<\infty$.  

Let us discuss here briefly the relation between the conditions (\ref{direct pos}) and (\ref{dual pos}) and the Sawyer type testing. The Sawyer type testing corresponds to the case when there is only one term in the sums in (\ref{direct pos}) and (\ref{dual pos}), that is the operator and its formal adjoint would be tested with one indicator of a dyadic cube at a time. Hence it is clear that the square function testing condition implies the Sawyer type testing condition.

On the other hand, when $p=q=2$, the left hand side of (\ref{direct pos}) can be written as
\begin{equation*}
\Big\| \Big( \sum_{Q \in \D_{0}} (a_{Q}T^{\mu}_{Q}1_{Q})^{2}\Big)^{\frac{1}{2}} \Big\|_{L^{2}(\nu)}
=\Big( \sum_{Q \in \D_{0}} \|a_{Q}T^{\mu}_{Q}1_{Q}\|_{L^{2}(\nu)}^{2}\Big)^{\frac{1}{2}},
\end{equation*}
and a similar computation on the right hand side of (\ref{direct pos}) shows that in this case the Sawyer type testing would imply the square function testing.

Equation (\ref{direct pos}) could be written with the Kahane-Khinchine inequalities
as 
\begin{equation*}
\E \Big\|\sum_{Q \in \D_{0}} \varepsilon_{Q}a_{Q}T^{\mu}_{Q}1_{Q} \Big\|_{L^{q}(\nu)} 
\leq C_{1}'\E \Big\|\sum_{Q \in \D_{0}} \varepsilon_{Q}a_{Q}1_{Q} \Big\|_{L^{p}(\mu)},
\end{equation*}
where the constants $C'_{1}$ and $C_{1}$ are comparable depending only on $p$ and $q$. This formulation explains how the square function testing is in the spirit of $\mathcal{R}$-bounded operator families, as mentioned in the introduction.

The Sawyer type testing is in general sufficient for (\ref{estimate pos}) if and only if the exponents are in the order $1<p\leq q\leq\infty$, see \cite{HHL}. Thus, in this situation our result is worse than the existing one. In the case $1<q<p<\infty$ H. Tanaka \cite{T} has given a characterization in terms of discrete Wolff's potentials, and here our result can be seen as an alternative way. We note that in our method the relative position of the exponents $p$ and $q$ does not make any difference. 

\begin{proof}[ Proof of theorem \ref{Pos oper}]
Our proof will follow the technique of  ``parallel stopping cubes'' as in \cite{H}. This method was first introduced in \cite{LSSU-T} (only in the older arXiv versions) and was used in the investigations of the two weight inequality for the Hilbert transform. 

If (\ref{estimate pos}) holds, then the sum (\ref{def of pos}) defining $T^{\mu}$ actually defines a bounded linear operator from $L^{p}(\mu)$ into $L^{q}(\nu)$. Clearly $T^{\mu}_{Q}1_{Q} \leq T^{\mu}1_{Q}$ for every $Q \in \D$, and the same is true for $T^{\nu}$ also. Hence in this situation we may apply Theorem \ref{vector extension} to show that (\ref{direct pos}) and (\ref{dual pos}) hold.

Assume then that (\ref{direct pos}) and (\ref{dual pos}) are true, and let $0 \leq f \in L^{p}(\mu)$ and $ 0\leq g \in L^{q'}(\nu)$ be two functions. For the estimate (\ref{estimate pos}) it is enough to choose a finite subcollection $\D_{0} \subset \D$ and show that
\begin{equation}\label{beginning}
\sum_{Q\in \D_{0}}\lambda_{Q}\int_{Q}f \ud \mu \int_{Q}g \ud \nu \lesssim (C_{1}+C_{2}) \|f\|_{L^{p}(\mu)}\|g\|_{L^{q'}(\nu)}.
\end{equation}

Since $\D_{0}$ is finite, we can construct the collections $\F$ and $\G$ of principal cubes for the function $f$ and $g$, respectively, where $\F$ is constructed with respect to the measure $\mu$ and $\G$ with respect to $\nu$. If $Q \in \D_{0}$ the notation $\pi Q=(F,G)$ means that $\pi_{\F}Q=F$ and $\pi_{\G}Q=G$. 

For every cube $Q \in \D_{0}$ there is a unique pair $(F,G) \in \F \times \G$ so that $\pi Q=(F,G)$, and the properties of dyadic cubes imply that $F \subset G$ or $G\subset F$. Hence, the sum in (\ref{beginning}) may be divided as 
\begin{equation}\label{division pos}
\sum_{Q \in \D_{0}}
\leq \sum_{F \in \F} \sum_{\begin{substack}{G \in \G: \\ G \subset F}\end{substack}}
\sum_{\begin{substack}{Q \in \D_{0}:\\ \pi Q=(F,G)}\end{substack}} 
+\sum_{G \in \G} \sum_{\begin{substack}{F \in \F :\\ F \subset G}\end{substack}}
\sum_{\begin{substack}{Q \in \D_{0}:\\ \pi Q=(F,G)}\end{substack}},
\end{equation}
where ``$\leq$'' is needed since we have double-counted the terms corresponding to all $Q$ for which $\pi_{\F}Q= \pi_{\G}Q$.  The two sums in (\ref{division pos}) are treated in the same way by symmetry, and we focus on the first one. 

So let $\F \ni F \supset G \in \G$ and suppose $Q \in \D_{0}$ is such that $\pi Q=(F,G)$. Write $ch_{\F}^{*}(F)$ for the collection of all $F' \in ch_{\F}(F)$ such that $\pi_{\G}F'\subset F$. Then, by the construction of $\F$, we have 
$ \langle f \rangle_{Q}^{\mu} \leq 2 \langle f \rangle_{F}^{\mu}$, and 
\begin{equation*}
\int_{Q} g \ud \nu= \int_{Q} g_{F} \ud \nu,
\end{equation*}
where 
\begin{equation*}
g_{F}:= 1_{E(F)} g + \sum_{F' \in ch_{\F}^{*}(F)}\langle g \rangle_{F'}^{\nu}.
\end{equation*}

Using these observations we get
\begin{equation*}
\begin{split}
&\sum_{F \in \F} \sum_{\begin{substack}{G \in \G: \\ G \subset F}\end{substack}}
\sum_{\begin{substack}{Q \in \D_{0}:\\ \pi(Q)=(F,G)}\end{substack}} 
\lambda_{Q}\int_{Q}f \ud \mu \int_{Q}g \ud \nu 
 \lesssim  \sum_{F \in \F} \langle f \rangle_{F}^{\mu} \sum_{\begin{substack}{Q \in \D: \\ Q \subset F}\end{substack}}\lambda_{Q}\int_{Q}1 \ud \mu \int_{Q}g_{F} \ud \nu \\
&= \int \sum_{F \in \F} \big(\langle f \rangle_{F}^{\mu} T^{\mu}_{F}1_{F}\big) g_{F} \ud \nu 
\leq \Big\| \Big(\ \sum_{F \in \F}\big(\langle f \rangle_{F}^{\mu} T^{\mu}_{F}1_{F}\big)^{2}\Big)^{\frac{1}{2}} \Big\|_{L^{q}(\nu)}
\Big\| \Big(\sum_{F \in \F} (g_{F})^{2}\Big)^{\frac{1}{2}} \Big \|_{L^{q'}(\nu)} \\
&\leq C_{1} \Big\| \Big(\ \sum_{F \in \F}\big(\langle f \rangle_{F}^{\mu}1_{F}\big)^{2}\Big)^{\frac{1}{2}}\Big\|_{L^{p}(\mu)}
\Big\| \Big(\sum_{F \in \F} (g_{F})^{2}\Big)^{\frac{1}{2}} \Big \|_{L^{q'}(\nu)}.
\end{split}
\end{equation*}

The Carleson embedding theorem \ref{Carleson} implies that 
\begin{equation*}
\Big\| \Big(\ \sum_{F \in \F}\big(\langle f \rangle_{F}^{\mu}1_{F}\big)^{2}\Big)^{\frac{1}{2}}\Big\|_{L^{p}(\mu)}
\leq \Big\| \ \sum_{F \in \F}\langle f \rangle_{F}^{\mu}1_{F} \Big\|_{L^{p}(\mu)} \lesssim \| f \|_{L^{p}(\mu)}.
\end{equation*}
Considering the second factor we have 
\begin{equation*}
\begin{split}
&\sum_{F \in \F} g_{F}= \sum_{F \in \F} (1_{E(F)}g +\sum_{F' \in ch_{\F}^{*}(F)}\langle g \rangle_{F'}^{\nu}1_{F'})
 \leq  g+ \sum_{F \in \F} \sum_{\begin{substack}{G \in \G: \\ \pi_{\F} G = F \text{ or} \\ G \in ch_{\F}(F) }\end{substack}} \sum_{\begin{substack}{F' \in ch_{\F}(F): \\ \pi_{\G} F'=G}\end{substack}} \langle g \rangle_{F'}^{\nu} 1_{F'} \\
& \leq g+ 2\sum_{F \in \F} \sum_{\begin{substack}{G \in \G: \\ \pi_{\F} G = F \text{ or} \\ G \in ch_{\F}(F)}\end{substack}} \langle g \rangle_{G}^{\nu} \sum_{\begin{substack}{F' \in ch_{\F}(F): \\ \pi_{\G} F'=G}\end{substack}} 1_{F'}
\leq g+ 2\sum_{F \in \F} \sum_{\begin{substack}{G \in \G: \\ \pi_{\F} G = F \text{ or} \\ G \in ch_{\F}(F)}\end{substack}} \langle g \rangle_{G}^{\nu}1_{G} \\
&\leq g+ 4\sum_{G \in \G}  \langle g \rangle_{G}^{\nu}1_{G}.
\end{split}
\end{equation*} 
Hence Theorem \ref{Carleson} implies again that
\begin{equation*}
\Big\| \Big(\sum_{F \in \F} (g_{F})^{2}\Big)^{\frac{1}{2}} \Big \|_{L^{q'}(\nu)}
\leq \Big\| \sum_{F \in \F} g_{F} \Big \|_{L^{q'}(\nu)}
\lesssim \|g \|_{L^{q'}(\nu)}.
\end{equation*}
\end{proof}

\section{Well localized operators}\label{Well localized operators}

We turn our attention to the main result of this paper and first recall the definition of a well localized operator from  \cite{NTV}. Suppose that we have a linear function $T^{\mu}$ mapping finite linear combinations of indicators $1_{Q}, Q \in \D$, into locally $\nu$-integrable functions. Also it is assumed that we have a linear $T^{\nu}$ mapping indicators $1_{Q}, Q \in \D$, into locally $\mu$-integrable functions so that
\begin{equation*}
\langle T^{\mu}1_{Q}, 1_{R}\rangle_{\nu}= \langle 1_{Q}, T^{\nu}1_{R}\rangle_{\mu}
\end{equation*}
for all $Q,R \in \D$. We call $T^{\nu}$ and $T^{\mu}$ formal adjoints of each other.

\begin{definition}\label{def well}
Fix some number $r \in \{0,1,2,\dots\}$. The operator $T^{\mu}$ is said to be \emph{lower triangularly localized} if
\begin{equation*}
\langle T^{\mu}1_{Q}, h^{\nu}_{R}\rangle_{\nu}=0
\end{equation*}
for all $Q,R \in \D$ such that 
\begin{itemize}
\item $l(R) \leq 2l(Q)$ and $R \not\subset Q^{(r+1)}$, or
\item $l(R) \leq 2^{-r}l(Q)$ and $R \not\subset Q$.
\end{itemize}
The operator $T^{\mu}$ is \emph{well localized} if both $T^{\mu}$ and $T^{\nu}$ are lower triangularly localized.  
\end{definition} 

\begin{remark}
This definition differs slightly from the definition given by Nazarov, Treil and Volberg in \cite{NTV}. The difference is that our condition ``$l(R) \leq 2l(Q)$ and $R \not\subset Q^{(r+1)}$'' above corresponds to ``$l(R) \leq l(Q)$ and $R \not\subset Q^{(r)}$'' in \cite{NTV}. This modification seems necessary to handle the sum $I$ in the proof of Theorem \ref{well loc} below. 
\end{remark}

A special example of a well localized operator is the two weight formulation of a Haar multiplier. Suppose we are on $\R$ for the moment and let $h_{I}:=|I|^{-1/2}(1_{I_{1}}-1_{I_{2}})$ be the $L^{2}$- normalized Haar function of a dyadic interval $I \in \D$. Let $\{\lambda_{I}\}_{I \in \D}$ be a set of real numbers such that only finitely many $\lambda_{I}$s are non-zero. Then consider the operator $T^{\mu}_{\lambda}$ defined for locally $\mu$-integrable functions by
\begin{equation*}
T^{\mu}_{\lambda}f:= \sum_{I \in \D}\lambda_{I} \langle f, h_{I} \rangle_{\mu}h_{I}.
\end{equation*}
We assumed finiteness of the coefficient sequence just to have the operator well defined in the general two weight setting, but of course it does not always have to be finite.

It is not difficult to see that in this case $T^{\mu}_{\lambda}$ is a well localized operator with parameter $r=0$. In a similar way one could see that dyadic shifts in $\R^{n}$ (of which the Haar multiplier is an example) in this two weight formulation would be well localized operators. For a discussion about where the motivation for the definition of a well localized operator comes from we refer to \cite{NTV}. There it is also shown that, more generally than just for dyadic shifts, two weight inequalities for the so called ``band operators'' can be seen as two weight inequalities for well localized operators.

The following main theorem characterizes the boundedness of well localized operators:

\begin{theorem}\label{well loc}
Let $1 < p,q < \infty$ be two exponents and suppose $T^{\mu}$ is a well localized operator with a formal adjoint $T^{\nu}$. The mapping $T^{\mu}$ extends to a bounded operator $T^{\mu}: L^{p}(\mu) \to L^{q}(\nu)$ if and only if there exist constants $C_{1},C_{2}>0$ such that for every finite subcollection $\D_{0} \subset \D$ and every set of non-negative real numbers $\{a_{Q}\}_{Q \in \D_{0}}$
the inequalities 

\begin{equation}\label{direct}
\Big\| \Big( \sum_{Q \in \D_{0}} ( 1_{R(Q)} T^{\mu}a_{Q}1_{Q})^{2} \Big)^{\frac{1}{2}} \Big \|_{L^{q}(\nu)} 
\leq
C_{1} \Big\| \Big( \sum_{Q \in \D_{0}} (a_{Q}1_{Q})^{2} \Big)^{\frac{1}{2}} \Big\|_{L^{p}(\mu)}
\end{equation}
and 
\begin{equation}\label{dual}
\Big\| \Big( \sum_{Q \in \D_{0}} ( 1_{R(Q)} T^{\nu}a_{Q}1_{Q})^{2} \Big)^{\frac{1}{2}}\Big\|_{L^{p'}(\mu)} 
\leq
C_{2} \Big\| \Big( \sum_{Q \in \D_{0}} (a_{Q}1_{Q})^{2} \Big)^{\frac{1}{2}} \Big\|_{L^{q'}(\nu)}
\end{equation}
hold. Here $R(Q) \in \D_{0}$ is any cube of  size $l(R(Q))=l(Q)$.

Furthermore in the case $T^{\mu}$ is bounded, its norm satisfies the estimate 
\begin{equation*}
\|T^{\mu}\|_{L^{p}(\mu) \to L^{q}(\nu)} \linebreak \simeq \mathcal{T}+\mathcal{T^{*}},
\end{equation*} 
where $\mathcal{T}$ and $\mathcal{T^{*}}$ denote the best constants in (\ref{direct}) and (\ref{dual}), respectively.

\end{theorem}

\begin{remark}
The testing conditions (\ref{direct}) and (\ref{dual}) will be applied in the situation where the cubes $R(Q)\in \D_{0}$ have side length $2l(Q)$. This is possible since every dyadic cube can be covered with its $2^{n}$ children.

The cubes $R(Q)$ appearing in the testing conditions can actually be assumed in a sense to be close to the cube $Q$. To be precise, when we use the full square function testing conditions (\ref{direct}) and (\ref{dual}), the cubes $R(Q)$ satisfy $l(R(Q))= l(Q)$ and $R(Q) \subset Q^{(r+1)}$, where $r$ is the parameter from the definition of the well localized operator.
 
 Also reduced versions of equations (\ref{direct}) and (\ref{dual}) where there is only one term in the sums, that is the Sawyer type testing, will be used. If the underlying dyadic system has an increasing sequence $R_{0}\subset R_{1}\subset \dots$ of cubes so that $\R^{n}=\cup_{k=0}^{\infty}R_{k}$, then Sawyer type testing will be used only with $Q=R(Q)$. 

But if the dyadic system does not have an increasing sequence of cubes covering the whole space, then the Sawyer type testing will be used when the cubes $Q$ and $R(Q)$ have equal side length $l(Q)=l(R(Q))$ and are adjacent in the sense that $\overline{Q} \cap \overline{R(Q)} \not=\emptyset$.  
\end{remark}

\begin{remark}
We explain here the consequence of F. Nazarov's counterexample mentioned in the introduction. Namely, for any exponent $1<p<\infty, \ p\not=2,$ the Sawyer type testing for Haar multipliers fails in the following quantitative sense: There \emph{does not} exist a universal constant $C$ such that for an arbitrary Haar multiplier $T^{\mu}$, with a formal adjoint $T^{\nu}$, the testing conditions
\begin{equation*}
\Big \|T^{\mu}1_{Q}\Big\|_{L^{p}(\nu)} \leq \mathcal{T} \mu(Q)^{\frac{1}{p}} \ \ (\text{for all } Q \in \D)
\end{equation*}
and
\begin{equation*}
\Big \| T^{\nu}1_{Q}\Big\|_{L^{p'}(\mu)} \leq \mathcal{T}^{*} \nu(Q)^{\frac{1}{p'}} \ \ (\text{for all } Q \in \D)
\end{equation*}
would imply that the operator $T^{\mu}$ could be extended to a bounded operator $T^{\mu}:L^{p}(\mu) \to L^{p}(\nu)$ with norm at most $C(\mathcal{T} + \mathcal{T}^{*})$. In other words, Theorem \ref{well loc} fails for general exponents $p$ if the square function testing is reduced to Sawyer type testing.
\end{remark}

\begin{proof}[Proof of Theorem \ref{well loc}]
That (\ref{direct}) and (\ref{dual}) hold if $T^{\mu}$ is bounded follows again from Theorem \ref{vector extension}, since the quantities on the left hand side of (\ref{direct}) and (\ref{dual}) can be made bigger by omitting the indicators $1_{R(Q)}$. Hence, we only need to prove sufficiency, which we show next.

We fix two compactly supported and bounded functions  $f \in L^{p}(\mu)$ and $g \in L^{q'}(\nu)$. There are at most $2^{n}$ big cubes $P_{1}, \dots, P_{2^{n}} \in \D$ that cover the supports of $f$ and $g$. Perform the martingale decompositions
\begin{equation*}
f= \sum_{i=1}^{2^{n}} \langle f \rangle^{\mu}_{P_{i}}1_{P_{i}} + \sum_{\begin{substack}{Q \in \D: \\ Q \subset \cup_{i=1}^{2^{n}} P_{i}}\end{substack}} \de^{\mu}_{Q} f
\end{equation*}
and
\begin{equation*}
g= \sum_{i=1}^{2^{n}} \langle g \rangle^{\nu}_{P_{i}}1_{P_{i}} + \sum_{\begin{substack}{Q \in \D: \\ Q \subset \cup_{i=1}^{2^{n}} P_{i}}\end{substack}} \de^{\nu}_{Q} g.
\end{equation*}
We may furthermore assume that there are only finitely many terms in the decompositions of $f$ and $g$, since these kind of functions are dense in $L^{p}$. Accordingly every sum below is finite, so there are no convergence problems. Also, for clear notation, define the functions 
$$ 
\tilde{f}:=f -\sum_{i=1}^{2^{n}}\langle f\rangle^{\mu}_{P_{i}}1_{P_{i}}, \ \ \ \tilde{g}:=g -\sum_{i=1}^{2^{n}}\langle g \rangle^{\nu}_{P_{i}}1_{P_{i}}.
$$

We consider the pairing $\langle T^{\mu}f, g \rangle_{\nu}$ and use the martingale difference decompositions to write it as
\begin{equation}\label{initial split}
\begin{split}
&\sum_{i=1}^{2^{n}} \Big\langle  \langle f\rangle^{\mu}_{P_{i}}T^{\mu}1_{P_{i}}, g \Big\rangle_{\nu} 
+\sum_{j=1}^{2^{n}} \Big\langle   \tilde{f},   \langle g\rangle^{\nu}_{P_{j}}T^{\nu}1_{P_{j}} \Big\rangle_{\mu} \\
&+ \sum_{Q,R \in \D} \Big\langle T^{\mu} \de^{\mu}_{Q} \tilde{f} ,  \de^{\nu}_{R} \tilde{g} \Big\rangle_{\nu}.
\end{split}
\end{equation}
Note that $\de^{\mu}_{Q} \tilde{f}$ can be non-zero only if $Q \subset \cup_{i=1}^{2^{n}} P_{i}$, and similarly with $\tilde{g}$.

As a direct application of the testing conditions (\ref{direct}) and (\ref{dual}) each term in the first two sums in (\ref{initial split}) is bounded by a testing constant times the norms of $f$ and $g$, and actually we need here only the Sawyer type testing. Since there are finitely many (depending on the dimension $n$) terms in those sums we see that they are bounded as we want. 

Turn attention to the third sum in  (\ref{initial split}). It is further split according to the relative size of the cubes $Q$ and $R$ as
\begin{equation*}
\sum_{\begin{substack}{Q,R \in \D: \\ l(R) \leq l(Q)}\end{substack}}\langle T^{\mu} \de^{\mu}_{Q} \tilde{f} , \ \de^{\nu}_{R} \tilde{g} \rangle_{\nu}
+\sum_{\begin{substack}{Q,R \in \D: \\ l(R) > l(Q)}\end{substack}} \langle  \de^{\mu}_{Q} \tilde{f} , \ T^{\nu}\de^{\nu}_{R} \tilde{g} \rangle_{\mu},
\end{equation*}
and by symmetry we concentrate on the first half.

So suppose we have two cubes $Q, R \in \D$ with $l(R) \leq l(Q)$. For $\langle T^{\mu} \de^{\mu}_{Q} \tilde{f} , \ \de^{\nu}_{R} \tilde{g} \rangle_{\nu}$ to be non-zero it must first of all be, because $T^\mu$ is well localized, that $R \subset Q^{(r)}$. Also, if $l(R) < 2^{-r}l(Q)$, then $R \subset Q$. Hence the summation we are considering can be split as
\begin{equation*}
\begin{split}
\sum_{ \begin{substack}{Q,R \in \D: \\ R \subset Q^{( r )}, \\ l(Q)\geq l( R ) \geq 2^{-r}l(Q)}\end{substack}}
\langle T^{\mu} \de^{\mu}_{Q} \tilde{f} , \ \de^{\nu}_{R} \tilde{g} \rangle_{\nu} 
+\sum_{ \begin{substack}{Q,R \in \D: \\ R \subset Q, \\ l(R) < 2^{-r}l(Q)}\end{substack}}
\langle T^{\mu} \de^{\mu}_{Q} \tilde{f} , \ \de^{\nu}_{R} \tilde{g} \rangle_{\nu} 
=: I+II.
\end{split}
\end{equation*}

\subsection*{Estimation of $I$}
For each $Q\in \D$ there are at most $N(r)$ cubes  $R$  such that $R \subset Q^{(r)}$ and $2^{-r}l(Q) \leq l(R) \leq l(Q)$.  Hence the sum $I$ may be divided into $N(r)$ different sums of the form
\begin{equation}\label{only one R}
\sum_{Q \in \D} \langle T^{\mu} \de^{\mu}_{Q} \tilde{f} , \ \de^{\nu}_{R(Q)} \tilde{g} \rangle_{\nu},
\end{equation}
where $R(Q) \in \D$ is a cube such that $R \subset Q^{(r)}$ and $2^{-r}l(Q) \leq l(R) \leq l(Q)$.
 
Let $Q^1,\dots,Q^{2^n}$ denote the children $ch(Q)$ of $Q$. Then concerning \eqref{only one R} we have

\begin{equation*}
\begin{split}
&\Big| \sum_{Q \in \D} \langle T^{\mu} \de^{\mu}_{Q} \tilde{f} , \ \de^{\nu}_{R(Q)} \tilde{g} \rangle_{\nu} \Big| \\
&\leq \Big\| \Big( \sum_{Q \in \D} ( 1_{R(Q)} | T^{\mu}  \de^{\mu}_{Q} \tilde{f} |^{2} \Big)^{\frac{1}{2}} \Big \|_{L^{q}(\nu)}
\Big\| \Big( \sum_{Q \in \D} (  |   \de^{\nu}_{R(Q)} \tilde{g} |^{2} \Big)^{\frac{1}{2}} \Big \|_{L^{q'}(\nu)} \\
& \lesssim \sum_{k=1}^{2^{n}} \Big\| \Big( \sum_{Q \in \D} ( 1_{R(Q)} | T^{\mu}  \langle \de^{\mu}_{Q} \tilde{f}\rangle^{\mu}_{Q^{k}} 1_{Q^{k}} |^{2} \Big)^{\frac{1}{2}} \Big \|_{L^{q}(\nu)} \| \tilde{g} \|_{L^{q'}(\nu)} \\
& \lesssim C_{1} \sum_{k=1}^{2^{n}} \Big\| \Big( \sum_{Q \in \D}   |   \langle \de^{\mu}_{Q} \tilde{f}\rangle^{\mu}_{Q^{k}} 1_{Q^{k}} |^{2} \Big)^{\frac{1}{2}} \Big \|_{L^{p}(\mu)} \| \tilde{g} \|_{L^{q'}(\nu)} \lesssim C_{1} \| \tilde{f} \|_{L^{p}(\mu)} \| \tilde{g} \|_{L^{q'}(\nu)} \\
&\lesssim C_1\|f \|_{L^{p}(\mu)} \|g \|_{L^{q'}(\nu)}.
\end{split}
\end{equation*}  

\subsection*{Estimation of $II$} Suppose we have two cubes $Q,R \in \D$ such that $R \subset Q$ and $l(R)<2^{-r}l(Q)$. Denote by $Q_R$ the child $Q' \in ch(Q)$ that contains $R$. Then, because $T^\mu$ is well localized, it holds that
$$
\langle T^{\mu} \de^{\mu}_{Q} \tilde{f} , \ \de^{\nu}_{R} \tilde{g} \rangle_{\nu}
= \langle \de^{\mu}_{Q}\tilde{f} \rangle_{Q_R}^{\mu} \langle T^{\mu} 1_{R^{(r)}} , \ \de^{\nu}_{R} \tilde{g} \rangle_{\nu}.
$$
Hence
\begin{equation}\label{to para}
\begin{split}
\sum_{ \begin{substack}{Q,R \in \D: \\ R \subset Q, \\ l(R) < 2^{-r}l(Q)}\end{substack}}
\langle T^{\mu} \de^{\mu}_{Q} \tilde{f} , \ \de^{\nu}_{R} \tilde{g} \rangle_{\nu}
&=\sum_{R \in \D} \sum_{\begin{substack}{Q \in \D: Q \supsetneq R^{(r)}}\end{substack}} \langle \de^{\mu}_{Q}\tilde{f} \rangle_{Q_R}^{\mu} \langle T^{\mu} 1_{R^{(r)}} , \ \de^{\nu}_{R} \tilde{g} \rangle_{\nu} \\
&=\sum_{R \in \D} \langle \tilde{f} \rangle^\mu_{R^{(r)}} \langle T^{\mu} 1_{R^{(r)}} , \ \de^{\nu}_{R} \tilde{g} \rangle_{\nu} \\
& = \sum_{R \in \mathscr{D}} \langle \tilde{f} \rangle^\mu_{R} \langle T^{\mu} 1_{R} , \ P^\nu_{R,r} \tilde{g} \rangle_{\nu},
\end{split}
\end{equation}
where
$$
P^\nu_{R,r}\tilde{g}:=\sum_{\begin{substack}{R' \in \mathscr{D}: \\ R'^{(r)}=R}\end{substack}} \de^{\nu}_{R'}\tilde{g}.
$$

Let $\mathscr{D}_0:= \{Q \in \mathscr{D}: Q \subset \bigcup_{i=1}^{2^n}P_i\}$. Construct the principal cubes $\mathscr{F}$ for the function $\tilde{f}$ in the collection $\mathscr{D}_0$ as in Subsection \ref{Principal and Carleson}. Note that if $\langle \tilde{f} \rangle^{\mu}_{Q}\not=0$ for some $Q \in \mathscr{D}$, then $Q \in \mathscr{D}_0$.

Suppose $F \in \mathscr{F}$ and $Q \in \mathscr{D}_0$ are such that $\pi_\F Q =F$. Then there exists a constant $c_Q \in [-2,2]$ so that $c_Q\langle \tilde{f}\rangle^{\mu}_Q = \langle |\tilde{f}|\rangle^\mu_F$. Note also, that if $Q,R \in \mathscr{D}$ are two cubes such that $Q \supset R$, then
$$
\langle T^{\mu} 1_{R} , \ P^\nu_{R,r} \tilde{g} \rangle_{\nu}
=\langle T^{\mu} 1_{Q} , \ P^\nu_{R,r} \tilde{g} \rangle_{\nu}
$$
because $T^\mu$ is well localized. 

Hence, continuing from \eqref{to para}, we have
\begin{equation*}
\begin{split}
\Big|\sum_{R \in \mathscr{D}}& \langle \tilde{f} \rangle^\mu_{R} \langle T^{\mu} 1_{R} ,  P^\nu_{R,r} \tilde{g} \rangle_{\nu} \Big|
=\Big|\sum_{F \in \F}  \langle |\tilde{f}| \rangle^\mu_{F}\sum_{\begin{substack}{R \in \mathscr{D}_0: \\ \pi_\F R=F}\end{substack}}  \langle  1_F T^{\mu} 1_{F} , c_R P^\nu_{R,r} \tilde{g} \rangle_{\nu} \Big| \\
&\leq \Big\| \Big( \sum_{F \in \F} |\langle |\tilde{f}| \rangle^\mu_{F}T^{\mu}1_F|^2  \Big)^\frac{1}{2}\Big\|_{L^q(\nu)}
\Big\| \Big( \sum_{F \in \F} |\sum_{\begin{substack}{R \in \mathscr{D}_0: \\ \pi_\F R=F}\end{substack}}  c_R P^\nu_{R,r} \tilde{g}|^2  \Big)^\frac{1}{2}\Big\|_{L^{q'}(\nu)}.
\end{split}
\end{equation*}
We note here that this trick of introducing the constants $c_Q$ is from \cite{LSSU-T}. 

From the testing condition \eqref{direct} and the Carleson embedding theorem \ref{Carleson} it follows that
\begin{equation*}
\begin{split}
\Big\| \Big( \sum_{F \in \F} |\langle |\tilde{f}| \rangle^\mu_{F}T^{\mu}1_F|^2  \Big)^\frac{1}{2}\Big\|_{L^q(\nu)}
&\leq C_1\Big\| \Big( \sum_{F \in \F} (\langle |\tilde{f}| \rangle^\mu_{F})^21_F  \Big)^\frac{1}{2}\Big\|_{L^p(\mu)} \\
&\leq C_1 \Big\| \sum_{F \in \F} \langle |\tilde{f}| \rangle^\mu_{F}1_F\Big\|_{L^p(\mu)}  \\
&\lesssim C_1 \|\tilde{f}\|_{L^p(\mu)} 
\lesssim C_1 \|f\|_{L^p(\mu)}.
\end{split}
\end{equation*}
With a similar computation as in the proof of Theorem \ref{vector extension}, the Kahane-Khinchine inequality \eqref{KK} and Equation \eqref{mart. norm} for computing $L^p$-norms with martingale differences give
\begin{equation*}
\begin{split}
\Big\| \Big( \sum_{F \in \F} |\sum_{\begin{substack}{R \in \mathscr{D}_0: \\ \pi_\F R=F}\end{substack}}  c_R P^\nu_{R,r} \tilde{g}|^2  \Big)^\frac{1}{2}\Big\|_{L^{q'}(\nu)}
&\simeq \E \Big\|  \sum_{F \in \F} \varepsilon_F \sum_{\begin{substack}{R \in \mathscr{D}_0: \\ \pi_\F R=F}\end{substack}}  c_R P^\nu_{R,r} \tilde{g} \Big\|_{L^{q'}(\nu)} \\
&\lesssim \|\tilde{g}\|_{L^{q'}(\nu)} \lesssim \|g\|_{L^{q'}(\nu)}.
\end{split}
\end{equation*}
This concludes the estimation of the sum $II$, and hence also the proof of Theorem \ref{well loc}.

\begin{remark}
Let $\mathscr{D}_0 \subset \D$ be any finite subcollection. Define a paraproduct operator by
$$
\Pi^{\mu,\nu}_{T,\D_{0}}f:=\sum_{Q \in \D_{0}} \langle f \rangle_{Q}^{\mu} \sum_{R \in ch^{(r)} Q} \de^{\nu}_{R} T^{\mu}1_{Q}.
$$
From the testing condition \eqref{direct} it follows, with similar arguments as in the estimation of $II$, that 
$$
\|\Pi^{\mu,\nu}_{T,\D_{0}}\|_{L^p(\mu) \to L^q(\nu)} \lesssim C_1,
$$ 
where the implied constant does not depend on the collection $\D_0$. Conversely, by slightly modifying the proof of Theorem \ref{well loc}, the part that corresponds to the estimation of the sum $II$ follows from this fact. See \cite{NTV} where this was done in the case $p=q=2$.
\end{remark}

\end{proof}

\end{document}